\documentclass[12pt, reqno]{amsart}

\usepackage{amssymb, fullpage}
\usepackage[colorlinks]{hyperref}
\hypersetup{citecolor=blue}
\usepackage[alphabetic]{amsrefs}

\numberwithin{equation}{section}
\theoremstyle{plain}
\newtheorem{thm}{Theorem}[section]
\newtheorem{cor}{Corollary}[section]
\newtheorem{lem}{Lemma}[section]
\newtheorem{prop}{Proposition}[section]

\theoremstyle{definition}
\newtheorem{defin}{Definition}[section]
\newtheorem{remark}{Remark}[section]

\newcommand{\I}{\mathcal I}
\newcommand{\R}{\mathbb R}
\newcommand{\Z}{\mathbb Z}
\newcommand{\p}{\mathbb P}
\newcommand{\E}{\mathbb E}
\newcommand{\T}{\mathbb T}
\newcommand{\eps}{\varepsilon}
\newcommand{\spn}{\mathrm{span}}
\newcommand{\dist}{\mathrm{dist}}
\newcommand{\rank}{\mathrm{rank}}

\begin{document}
\title{Inverse Littlewood-Offord problems for quasi-norms}

\author{Omer Friedland}
\address{Institut de Math\'ematiques de Jussieu, Universit\'e Pierre et Marie Curie \\ 4 Place Jussieu, 75005 Paris, France}
\email{omer.friedland@imj-prg.fr}

\author{Ohad Giladi}
\address{School of Mathematical and Physical Sciences, University of Newcastle \\  Callaghan, NSW 2308, Australia}
\email{ohad.giladi@newcastle.edu.au}

\author{Olivier Gu\'edon}
\address{Laboratoire d'Analyse et Math\'ematiques Appliqu\'ees, Universit\'e Paris-Est \\ 77454 Marne-la-Vall\'ee, France}
\email{olivier.guedon@u-pem.fr}

\begin{abstract}
Given a star-shaped domain $K\subseteq \R^d$, $n$ vectors $v_1,\dots,v_n \in \R^d$, a number $R>0$, and i.i.d. random variables $\eta_1,\dots,\eta_n$, we study the geometric and arithmetic structure of the set of  vectors $V = \{v_1,\dots,v_n\}$ under the assumption that the small ball probability
\begin{align*}
\sup_{x\in \R^d}~\p\Bigg(\sum_{j=1}^n\eta_jv_j\in x+RK\Bigg)
\end{align*}
does not decay too fast as $n\to \infty$. This generalises the case where $K$ is the Euclidean ball, which was previously studied in~\cite{NV11,TV12}.
\end{abstract}


\date{\today}

\maketitle

\section{Introduction}
\subsection{Background}\label{sec background}
A body $K\subseteq \R^d$ is said to be a star-shaped domain if for every $x\in K$, $tx\in K$ for every $t\in [0,1]$. In this note, $K$ will always assumed to be compact. Given a random vector $X$ in $\R^d$ and $R>0$, define the small-ball probability
\begin{align}\label{def rho}
\rho_R^K(X) = \sup_{x\in \R^d}\p\big(X\in x+RK\big).
\end{align}
In particular, if $V= \{v_1,\dots,v_n\}\subseteq\R^d$ is a set of $n$ fixed vectors, $\eta_1,\dots,\eta_n$ are i.i.d. random variables, then one can consider the following random vector,
\begin{align}\label{def xv}
X_V = \sum_{j=1}^n\eta_jv_j.
\end{align}
It is known that the asymptotic behaviour of $\rho_R^{B_2^d}(X_V)$ as $n\to \infty$ is closely related to the various structural aspects of the set $V$. Here and in what follows, $B_2^d$ denotes the Euclidean ball in $\R^d$. We refer the reader to~\cite{Erd45, FF88, NV11, NVSurvey, RV08, TV09, TV10, TV12} to name just a few, where this type of questions is discussed, as well as some interesting applications. In particular, we refer the reader to~\cite{NV11}, which includes some enlightening remarks and examples of the relation between the behaviour of $\rho_R^{B_2^d}(X_V)$ and and the structure of $V$, as well as to~\cite{NVSurvey}, which gives a broad introduction to the topic.

In the results of~\cite{NV11, TV12}, one always assumes that the norm on $\R^d$ is the Euclidean norm. One of the key technical tools in the proofs is Esseen type estimates, which relate the small ball probability to the behaviour of the characteristic function of $X_V$. See for example~\cite{Ess66} and~\cite[Section~7.3]{TVBook}. Esseen's inequality for a general random vector $X$ says that for every $\eps>0$,
\begin{align}\label{esseen original}
\rho_R^{B_2^d}(X) \le C^d \left(\frac R{\sqrt d} + \frac{\sqrt d}{\eps}\right)^d\int_{\eps B_2^d}\big| \E \exp\left(i \langle X, \xi \rangle\right)\big|d\xi.
\end{align}
In~\eqref{esseen original} and in what follows, $C$ denotes an absolute constant. In~\cite{FGG14}, based on previous work from~\cite{FG11}, an Esseen type estimate was obtained for a general quasi-norm. If $K\subseteq \R^d$ is a centrally symmetric star-shaped domain in $\R^d$, then the functional 
\begin{align}\label{def norm k}
\|x\|_K = \inf\big\{t>0 ~\big|~ x \in tK\big\},
\end{align} 
is a quasi-norm, that is, $\|\cdot\|_K$ behaves like a norm, with the only exception that instead of the triangle inequality, there exists a number $C_K\ge 1$ such that for every $x,y\in \R^d$, 
\[\|x+y\|_K \le C_K\big(\|x\|_K+\|y\|_K\big).\] 
The case $C_K=1$ corresponds to the case when $\|\cdot\|_K$ is a norm and $K$ is convex. If we omit the assumption that $K$ is centrally symmetric then we do not have $\|x\|_K = \|-x\|_K$. In this note we do not need to assume that $K$ is centrally symmetric. The following Esseen type estimate was shown in \cite{FGG14}. 
\begin{align}\label{esseen k}
\rho_R^K(X) & \le \left(\kappa(K)R\right)^d\int_{\R^d}\big| \E \exp(i\langle X, \xi \rangle )\big|e^{-\frac{R^2|\xi|_2^2}{2}}d\xi = \kappa(K)^d\cdot \I\left(\frac 1 R \, X\right),
\end{align}
where we deonte
\begin{align*}
\kappa(K) = C_K\sqrt{\frac{2}{\pi}} \left(\frac{\mu_d(K)}{\gamma_d(K)}\right)^{1/d}, \quad \I(X) = \int_{\R^d}\big|\E\exp\big(i\langle X,\xi\rangle\big)\big|e^{-\frac{|\xi|_2^2}{2}}d\xi,
\end{align*}
$\gamma_d(K)$ being the $d$-dimensional gaussian measure of $K$, and $\mu_d(K)$ its Lebesgue measure. In particular, if $X=X_V$ as defined in~\eqref{def xv}, we have
\begin{eqnarray}\label{int v}
\nonumber \I(X_V) & = & \int_{\R^d}\left| \E \exp\left(i\left\langle \sum_{i=1}^n\eta_jv_j, \xi \right\rangle \right)\right|e^{-\frac{|\xi|_2^2}{2}}d\xi
\\ & \stackrel{(*)}{=} &\int_{\R^d}\left[\prod_{j=1}^n\Big| \E \exp\big(i\left\langle \eta_jv_j, \xi \right\rangle \big)\Big|\right]e^{-\frac{|\xi|_2^2}{2}}d\xi,
\end{eqnarray}
where in ($*$) we used the fact the $\eta_j$'s are independent. Inequality~\eqref{esseen k}, as well as inequality~\eqref{esseen original}, imply that there is a relation between the behaviour of $\rho_R^K(X_V)$ and the arithmetic behaviour of the vector $X_V$. Note also that~\eqref{int v} implies that it is natural to consider random variables $\eta_j$'s that satisfy some anti-concentration property. See Section~\ref{sec hyper} and Section~\ref{sec gap}, and in particular the anti-concentration conditions~\eqref{growth eta} and~\eqref{assume eta}. Therefore, given~\eqref{esseen k} and~\eqref{int v}, it is natural to consider the following type of problems, also known as \emph{Inverse Littlewood-Offord Problems}:

\medskip

\begin{center}
\emph{Assume that $\rho_R^K(X_V)$ is large. Show that the set $\{v_1,\dots,v_n\}\subseteq \R^d$ is well-structured.}
\end{center}

\medskip

Clearly, the term `large' should be formulated quantitatively, and the term `well-structured' can have different meanings. In this note, we discuss two ways to obtain `well-structured' sets. One way is to consider sets whose elements are all found near a given subspace of $\R^d$ (again, the term `near' can be made precise). In Section~\ref{sec hyper}, we show that if $\rho_R^K(X_V)$ does not decay too fast as $n\to \infty$, then many of the vectors in the set $\{v_1\,\dots,v_n\}\subseteq \R^d$ are `well-concentrated' around a given hyperplane. See Section~\ref{sec hyper} for the exact formulation. Then, in Section~\ref{sec gap}, we show that if $\rho_R^K(X_V)$ does not decay too fast, then the set $\{v_1,\dots,v_n\}\subseteq \R^d$ can be approximated with a set which has some arithmetic structure. See Section~\ref{sec gap} for the exact definitions and formulation. Finally, Section~\ref{sec proof hyper} and Section~\ref{sec proof gap} are dedicated to the proofs of the main theorems.

One point which is worth emphasising is the following. In the study of many asymptotic problems, including the ones discussed in~\cite{NV11, TV12}, one is primarily interested in the asymptotic behaviour as $n\to \infty$. In particular, since all norms in $\R^d$ are equivalent, any Euclidean result trivially yields a result for a general norm. For a quasi-norm, trivial bounds can also be easily obtained. For a quasi-norm, trivial bounds can also be deduced from the Euclidean results. The main purpose of this note is to obtain an estimate which is better than these trivial conclusions and to extend the results of~\cite{NV11, TV12} to a non-Euclidean setting. See Section~\ref{sec compare}, for a comparison of the previously obtained results with the results of this note.

\subsection*{Notations}
For a star-shaped body $K\subseteq \R^d$, we let $\|\cdot\|_K$ be defined as in~\eqref{def norm k}. In the special case of the $\ell_p^d$ norm, for $p\in (0,\infty]$ we denote
\begin{align}\label{def lp}
|x|_p = \|x\|_{B_p^d} = \left(\sum_{j=1}^d|x_j|^p\right)^{1/p}.
\end{align}
Note that if $p\ge 1$,~\eqref{def lp} gives a norm and for $p\le 1$,~\eqref{def lp} gives a quasi-norm with $C_{B_p^d} = 2^{1/p-1}$. 

For a set $S\subseteq \R^d$ and a vector $v\in \R^d$, denote
\begin{align*}
\mathrm{dist}_K(v,S) = \inf\big\{ \|x-s\|_K ~\big|~ s\in S\big\}.
\end{align*}
In particular, $\mathrm{dist}_2(v,S) = \mathrm{dist}_{B_2^d}(v,S)$, and $\mathrm{dist}_{\infty}(v,S) = \mathrm{dist}_{B_{\infty}^d}(v,S)$. 

Given a star-shaped domain $K\subseteq \R^d$ and $p\in (0,\infty]$, denote
\begin{align}\label{def omega}
\omega_p(K) = \inf\big\{t>0~\big|~ B_p^d\subseteq tK\big\},
\end{align}
and also
\begin{align}\label{def w}
\frac 1 {W_p(K)} = \sup\big\{t>0~\big|~ tK\subseteq B_p^d\big\}.
\end{align}
Note that since we have 
\begin{align*}
\frac 1 {\omega_p(K)}B_p^d \subseteq K \subseteq W_p(K) B_p^d,
\end{align*}
it follows that for every $x\in \R^d$,
\begin{align}\label{equiv norm}
\frac 1 {W_p(K)}|x|_p \le \|x\|_K \le \omega_p(K)|x|_p.
\end{align}

In this note, $C$ always denotes an absolute constant. If an implied constant depends on a parameter, say $\gamma$, we write $C(\gamma)$. Also, if $F$ is a finite set and $k$ is a positive integer, denote 
\[kF = \Big\{\sum_{j=1}^kv_j~\Big|~ v_j\in F\Big\}.\]
If $\alpha$ is a real number which is not an integer, then $\alpha F$ denotes the dilation of $F$, that is $\alpha F = \big\{\alpha \, x~\big|~x\in F\big\}$. $|F|$ denotes the cardinality of any finite set $F$. If $F$ is any set, for example, if $F$ is a star-shaped body, then $\mu_d(F)$ denotes its Lebesgue measure, while $\gamma_d(F)$ denotes its $d$-dimensional gaussian measure.

\section{Statement of the main results}
\subsection{Concentration near a hyperplane}\label{sec hyper}

The first result in this note shows that if the concentration function $\rho_R^K(X_V)$ is asymptotically large, then many vectors are necessarily close to a given hyperplane in $\R^d$. We begin by fixing some notation. For a real number $a$, let
\begin{align}\label{def norm}
\|a\|_{\T} = \inf_{z\in \Z} \big[\,|a-\pi z|\,\big],
\end{align}
where $\T=\R/\pi\Z$. As mentioned above, from \eqref{int v} it is natural to assume some bound on $\big| \E \exp\big(i\langle \eta_j v_j, \xi\rangle\big)\big|$. For the first theorem, we will use the following condition. There exists a number $c_{\eta}>0$ such that for every $a\in \R$, we have
\begin{align}\label{growth eta}
\big|\E\exp(i\eta a)\big| \le \exp\left(-c_{\eta}\|a\|_{\T}^2\right).
\end{align}
Condition \eqref{growth eta} can be thought of as an anti-concentration assumption. Note that, for example, symmetric Bernoulli random variables satisfy \eqref{growth eta}, since in this case we have
\begin{align*}
\big|\E\exp\big(i\eta a\big)\big| = |\cos(a)| \le 1-\frac 2 {\pi^2}\|a\|_{\T}^2 \le \exp\left(-\frac 2{\pi^2}\|a\|_{\T}^2\right).
\end{align*}
The main tool in the proof of Theorem~\ref{thm hyper} is the following proposition.

\begin{prop}\label{prop hyper}
Let $k\le n$ be integers. Let $V = \{v_1,\dots,v_n\} \subseteq \R^d$ be a set of fixed vectors. Assume that $\eta_1,\dots,\eta_n$ are i.i.d. random variables that satisfy \eqref{growth eta}. Assume also that for every hyperplane $H\subseteq \R^d$, there exists at least $n-k$ vectors satisfying $\mathrm{dist}_2(v_j,H) \ge R$. Then
\begin{align}\label{bound on i}
\I(X_V) \le \left(80\,\frac{R+1}{R}\,\sqrt{\frac{d}{d+c_{\eta}k}}\,\right)^d.
\end{align}
\end{prop}

The main result of this section is the following.

\begin{thm}\label{thm hyper}
Let $V = \{v_1,\dots,v_n\}\subseteq \R^d$ and $\eta_1,\dots, \eta_n$ be i.i.d. random variables satisfying \eqref{growth eta}. Assume that there exists $k \le n$ such that
\begin{align*}
\rho_R^K(X_V) \ge  \big(40\kappa(K))^d\left(\frac{d}{d+c_{\eta}k}\right)^{d/2}.
\end{align*}
Then there exists a hyperplane $H$ in $\R^d$ for which at least $n-k$ vectors from $V$ satisfy
\begin{align*}
\mathrm{dist}_2(v_j, H) \le R.
\end{align*}
In particular, using~\eqref{equiv norm}, we have 
\begin{align}\label{close in k}
\mathrm{dist}_K(v_j, H) \le \omega_2(K)R.
\end{align}
\end{thm}
Theorem~\ref{thm hyper} implies the following.

\begin{cor}
Let $A>0$ be a positive constant. Assume that for all $n$ sufficiently large, we have
\[\rho_R^K(X_V) \ge n^{-A}.\]
Then there exist at least $n-k$ vectors in $V$ satisfying 
\begin{align*}
\mathrm{dist}_K(v_j, H) \le \omega_2(K)R,
\end{align*}
and $k$ satisfies
\begin{align*}
k \le Cd\,\frac{\kappa(K)^2n^{2A/d}-1}{c_{\eta}}.
\end{align*}
\end{cor}

\begin{remark}
The Euclidean case of Theorem~\ref{thm hyper} is a key ingredient in the proof of the main theorem of~\cite{TV12}. More specifically, assuming that the $\eta_j$'s are Bernoulli random variables and defining
\begin{align}\label{defin p}
\mathcal P_R^{B_2^d}(n) = \sup\Big\{\rho_R^{B_2^d}(V)~\Big|~ V\subseteq \R^d,~|V|=n, |v|_2\ge 1 ~\forall v\in V\Big\},
\end{align}
the authors prove that
\begin{align}\label{sharp LO}
\mathcal P_R^{B_2^d}(n) = \big(1+o(1)\big)2^{-n}S(n,\lfloor R\rfloor +1),
\end{align}
where $S(n,m)$ is the sum of the $m$ largest binomial coefficients $n\choose \cdot$. Here the error term tends to 0 as $n\to \infty$. The authors also show that if $R$ is sufficiently close to an integer, then the error term in~\eqref{sharp LO} can be removed. This problem had previously been studies in the one-dimensional case in~\cite{Erd45} and in the multi-dimensional case in~\cite{FF88} (again with the Euclidean norm). Similarly to~\eqref{defin p}, one could define
\begin{align*}
\mathcal P_R^{K}(n) = \sup\Big\{\rho_R^{K}(V)~\Big|~ V\subseteq \R^d,~|V|=n, \|v\|_K\ge 1 ~\forall v\in V\Big\},
\end{align*}
and ask whether an estimate similar to~\eqref{sharp LO} could be obtained in some non-Euclidean setting. However, the proof of~\eqref{sharp LO} in~\cite{TV12} makes heavy use of the rotation invariance of the Euclidean norm, and therefore it is not clear how~\eqref{sharp LO} could be generalised.
\end{remark}

\subsection{Approximate arithmetic progression}\label{sec gap}
We begin with the following definition.
\begin{defin}[General arithmetic progression, GAP]\label{def gap}
A set $Q\subseteq \R^d$ is said to be general arithmetic progression (GAP), if there exist integers $r$, $L_1,\dots,L_r$ and vectors $g_1,\dots,g_r \in \R^d$ such that $Q$ can be written in the following way.
\[Q = \Bigg\{\sum_{j=1}^rx_jg_j ~\Bigg|~ x_j\in \Z, ~ |x_j|\le L_j, ~ j \le r\Bigg\}.\]
The number $r$ is said to be the rank of $Q$, and is deonted by $\rank(Q)$. $Q$ is said to be proper if we have
\[|Q| = \prod_{j=1}^r L_j.\]
Finally, the vectors $g_1,\dots,g_r\in \R^d$ are said to be generators of $Q$.
\end{defin}

\begin{remark}
For every GAP, we have $|Q| \le \prod_{j=1}^r L_j$. A GAP which is proper is a GAP in which no cancellation between the generators occurs.
\end{remark}

A set which is GAP clearly has an additive structure. Hence, in the context of Littlewood-Offord problem, one could expect that if $\rho_R^K(X_V)$ does not decay too fast as $n\to \infty$, then $V$ should have additive structure, which is given by Definition~\ref{def gap}. This problem has been studied in~\cite{NV11}. Here we consider the non-Euclidean setting.

As in Theorem~\ref{thm hyper}, we need some anti-concentration condition to assure that we get efficient bounds in \eqref{esseen k}. Here we use the following: that if $\eta_1, \eta_2$ are independent copies of a random variable $\eta$, then there exists a number $C_{\eta}>0$ such that
\begin{align}\label{assume eta}
\p\Big(1\le |\eta_1-\eta_2| \le C_{\eta}\Big) \ge \frac 1 2.
\end{align}
Note that Bernoulli variables satisfy \eqref{assume eta}, for example with $C_{\eta}=2$. We can now state the second main result of this note.

\begin{thm}\label{thm gap}
Fix absolute positive real numbers $A$ and $\eps$. Let $K$ be a star-shaped domain in $\R^d$, and let $\eta_1,\dots,\eta_n$ be i.i.d. random variables that satisfy~\eqref{assume eta}. Assume that
\[\rho^K_R(X_V) \ge n^{-A}.\] 
Let $n' \in [n^{\eps}, n]$ be a positive integer, and assume that $n$ is sufficiently large compared to $d$, $A$, $\eps$ and $\kappa(K)$. Then there exists a GAP $Q\subseteq \R^d$, a positive integer $k$ satisfying
\begin{align*}
\sqrt{\frac{n'}{640\pi^2\sqrt{d\log\left(n^A\kappa(K)\right)}}} \le k \le \sqrt{n'},
\end{align*}
and a number $\alpha $ which depends only on the constant $C_\eta$ from~\eqref{assume eta}, such that
\begin{enumerate}
\item\label{part 1} {\bf $Q$ has small rank and cardinality:}
\begin{align*}
\rank(Q) & \le C\left(d+\frac A{\eps}\right), 
\\ |Q| & \le C(A,d, \eps) \frac{(n')^{\frac{d-\rank(Q)}{2}}}{\rho_R^K(X_V)}.
\end{align*}
\item\label{part 2} {\bf $Q$ approximates $V$ in the $K$ quasi-norm:} At least $n-n'$ elements of $v\in V$ satisfy
\[\mathrm{dist}_{K}(v,Q) \le C(\eta)\frac{\omega_{\infty}(K)R}{dk}.\]
\item \label{part 3} {\bf $Q$ has full dimension:} There exists $C' \le Cd\alpha$ such that
\[\{-1,1\}^d \subseteq \frac {C'k} RQ.\] 
\item \label{part 4} {\bf The generators of $Q$ have bounded $K$ quasi-norm:}
\begin{align}\label{bound norm generator}
\max_{1\le j \le r}\|g_j\|_K \le C(A,d,\eps)\, C_K^{k+1}\left(\frac{d\alpha k}{R}\max_{v\in V}\|v\|_K+\omega_{\infty}(K)\right).
\end{align}
\end{enumerate}
\end{thm}

\begin{remark}
Note that when $K$ is not convex, that is, when $\|\cdot\|_K$ is a quasi-norm but not a norm, we have $C_K>1$, in which case~\eqref{bound norm generator} does not give a sublinear bound (in $n$) on the norm.
\end{remark}

\subsection{Comparing previous and new results}\label{sec compare}

As discussed in Section~\ref{sec background}, the main purpose of this note is to show that in some cases, one can obtain estimates which are better than estimates which are trivially obtained from using the results in the Euclidean setting. This is true for both Theorem~\ref{thm hyper} and for Theorem~\ref{thm gap}.  Recall that by~\eqref{def omega} and~\eqref{def w}, we have that
\begin{align}\label{inclusion}
\frac 1 {\omega_2(K)}B_2^d \subseteq K \subseteq W_2(K) B_2^d.
\end{align}
Now,~\eqref{inclusion} implies that 
\[\rho_R^K(X_V) \le \rho_{W_2(K)R}^{B_2^d}(X_V).\] 
If, in addition, we use the fact that $\|\cdot\|_K \stackrel{\eqref{equiv norm}}{\le} \omega_2(K)|\cdot|_2$, we can use the Euclidean version of Theorem~\ref{thm hyper} to conclude that
\begin{align}\label{using Euclidean}
d_K(v_j,H) \le \omega_2(K)W_2(K)R.
\end{align}
If, for example, we assume that $K$ is convex, that is, $\|\cdot\|_K$ is a norm, then by taking a linear transformation of $K$, we may assume that the Euclidean unit ball is the ellipsoid of maximal volume contained in $K$, in which case $B_2^d \subseteq K \subseteq \sqrt d B_2^d$, see for example~\cite{Bal97}. This implies that we have $\omega_2(K) = 1$ and $W_2(K) \le \sqrt d$. Thus, in general,~\eqref{using Euclidean} can be a worse bound than~\eqref{close in k}.

Similarly, by using the Euclidean version of Theorem~\ref{thm gap}, if we assume that $\rho_R^K(X_V) \ge n^{-A}$ then we have $\rho_{\frac{R}{W_2(K)}}^{B_2^d}(X_V) \ge n^{-A}$. Then using the Euclidean version of the theorem gives an approximating GAP, but in this case, by Part~\ref{part 2} and the fact that $\omega_{\infty}(B_2^d) = \sqrt d$, the Euclidean approximation is 
\begin{align*}
|v-q|_2 \le \frac{CW_2(K)R}{\sqrt d\, k},
\end{align*}
where $v\in V$ and $q\in Q$. Again, we have that $\|\cdot\|_K \stackrel{\eqref{equiv norm}}{\le} \omega_2(K)|\cdot|_2$, which means that the approximation in the $K$ norm is of order $\frac{\omega_2(K)W_2(K) R}{\sqrt d\, k}$. On the other hand, the approximation obtained from directly using Theorem~\ref{thm gap}, is of order $\frac{\omega_{\infty}(K)R}{dk}$. If again we assume that $K$ is in a position such that $B_2^d \subseteq K\subseteq \sqrt d B_2^d$, then we have that $\omega_2(K)W_2(K) \in \big[1,\sqrt d\,\big]$, while $\frac{\omega_{\infty}(K)}{\sqrt d} \le 1$. This means that the bound obtained in Part~\ref{part 2} of Theorem~\ref{thm gap} is generally better. Note however that if $K=B_2^d$ the two bounds coincide.

\begin{remark}
For every $t>0$, we have $\rho_R^{K}(V) = \rho_{R/t}^{tK}(V)$. Thus,~\eqref{esseen k} gives 
\begin{align*}
\rho_R^K(V) \le \inf_{t>0}\left[\kappa(tK)^d\cdot \mathcal I\left(\frac t R X_V\right)\right].
\end{align*}
Therefore, in order to find good bounds on $\rho_R^K(V)$, one possible approach would be to study the behaviour of $\kappa(tK)$, where $t>0$. Note that in the case $K$ is convex, that is, when $\|\cdot\|_{K}$ is a norm, the results of~\cite{CFM04} imply that $\kappa(tK) = \sqrt{\frac 2{\pi}}\, t\,e^{\varphi_K(t)}$, where $\varphi_K$ is convex. However, in general we do not seem to have enough information about $\varphi_K$ to obtain meaningful results. Also, it could be of interest to study bodies for which $\kappa(K)$ is a constant, that is, does not depend on $d$. By~\eqref{esseen k}, this would again yield good bounds on $\rho_R^K(V)$.
\end{remark}

\section{Proof of Proposition~\ref{prop hyper} and Theorem~\ref{thm hyper}}\label{sec proof hyper}

We begin with the following lemma, which is a simple variant of a result that appeared in \cite{TV12}.
\begin{lem}\label{lem TV}
Let $\lambda>0$ and $w\neq 0$. Then for every $\alpha \in \R$, we have
\begin{align*}
\int_{\R}\exp\left(-\lambda \|\xi w+\alpha\|_{\T}^2-\frac{|\xi|_2^2}{2}\right)d\xi \le \frac{40(|w|+1)}{|w|\,\sqrt {1+\lambda}}.
\end{align*}
\end{lem}

\begin{proof}
Since by definition~\eqref{def norm}, for every real number $w$ we have $\|w\|_{\T} = \|-w\|_{\T}$, we may assume without loss of generality that $w>0$. Using the change of variables $t = \xi w+\alpha$ and the fact that $w >0$, we get
\begin{align}\label{what we want with w}
\int_0^{\pi}\exp\left(-\lambda \|\xi w+\alpha\|_{\T}^2\right)d\xi & = \frac 1 {w}\int_{0 \le t-\alpha \le \pi w}\exp\left(-\lambda \|t\|_{\T}^2\right)dt.
\end{align}
Let $N = \lfloor w\rfloor+1$. Then $w \le N\le w+1$, and so we have
\begin{align}\label{split int}
\nonumber & \int_{0\le t-\alpha \le \pi w }\exp\left(-\lambda \|t\|_{\T}^2\right)dt \le \int_{0\le t-\alpha \le \pi N}\exp\left(-\lambda\|t\|_{\T}^2\right)dt
\\
\nonumber& = \sum_{s=0}^{N-1}\int_{\pi s+\alpha}^{\pi(s+1)+\alpha}\exp\left(-\lambda \|t\|_{\T}^2\right)dt \le N\max_{0\le s\le N-1}\left[\int_{\pi s+\alpha}^{\pi(s+1)+\alpha}\exp\left(-\lambda\|t\|_{\T}^2\right)dt\right]
\\
& \le (w +1\big)\max_{0\le s\le N-1}\left[\int_{\pi s+\alpha}^{\pi(s+1)+\alpha}\exp\left(-\lambda \|t\|_{\T}^2\right)dt\right].
\end{align}
Plugging \eqref{split int} into \eqref{what we want with w},
\begin{align}\label{bound with max}
\int_0^{\pi}\exp\left(-\lambda \|\xi w+\alpha\|_{\T}^2\right)d\xi \le  \frac{w +1}w\max_{0\le s\le N-1}\left[\int_{\pi s+\alpha}^{\pi(s+1)+\alpha}\exp\left(-\lambda \|t\|_{\T}^2\right)dt\right].
\end{align}
Consider first the integral
\begin{align*}
\int_0^{\pi}\exp\left(-\lambda\|t\|_{\T}^2\right)dt.
\end{align*}
This integral is trivially bounded by $\pi$. Also,
\begin{align*}
\int_{0}^{\pi}\exp\left(-\lambda \|t\|_{\T}^2\right)dt & = \int_{0}^{\pi/2}\exp\left(-\lambda \|t\|_{\T}^2\right)dt+\int_{\pi/2}^{\pi}\exp\left(-\lambda \|t\|_{\T}^2\right)dt
\\
& = \int_{0}^{\pi/2}\exp\left(-\lambda t^2\right)dt + \int_{\pi/2}^{\pi}\exp\left(-\lambda |t-\pi|^2\right)dt
\\
& = 2\int_{0}^{\pi/2}\exp\left(-\lambda t^2\right)dt
= \frac{2}{\sqrt{\lambda}}\int_{0}^{\frac{\lambda\pi^2}{4}}\frac{e^{-x}}{\sqrt x}dx \le \frac{6}{\sqrt{\lambda}}.
\end{align*}
Altogether, we get
\begin{align}\label{bound zero pi}
\int_{0}^{\pi}\exp\left(-\lambda \|t\|_{\T}^2\right)dt \le \min\left\{\pi,\frac{6}{\sqrt{\lambda}}\right\} \le \frac{10}{\sqrt{1+\lambda}}.
\end{align}
Since the function $\|\cdot\|_{\T}^2$ is $\pi$-periodic, it follows that for every $\alpha \in \R$,
\begin{align}\label{bound with s}
\nonumber \int_{\pi s+\alpha}^{\pi(s+1)+\alpha}\exp\left(-\lambda \|t\|_{\T}^2\right)dt & = \int_{\alpha}^{\pi+\alpha}\exp\left(-\lambda \|t\|_{\T}^2\right)dt  \le \int_{0}^{2\pi}\exp\left(-\lambda \|t\|_{\T}^2\right)dt
\\ & = 2\int_{0}^{\pi}\exp\left(-\lambda \|t\|_{\T}^2\right)dt \stackrel{\eqref{bound zero pi}}{\le} \frac{20}{\sqrt{1+\lambda}}.
\end{align}
Plugging \eqref{bound with s} into \eqref{bound with max}, we get
\begin{align*}
\int_0^{\pi}\exp\left(-\lambda \|\xi w+\alpha\|_{\T}^2\right)d\xi \le \frac{20\big(w +1\big)}{w\,\sqrt {1+\lambda}}.
\end{align*}
Since $\|\cdot\|_{\T}$ is $\pi$-periodic, we also have for every $s\in \Z$,
\begin{align*}
\int_{\pi s}^{\pi(s+1)}\exp\left(-\lambda \|\xi w+\alpha\|_{\T}^2\right)d\xi \le \frac{20(w +1\big)}{w\,\sqrt {1+\lambda}}.
\end{align*}
Hence,
\begin{align*}
& \int_{\R} \exp\left(-\lambda \|\xi w+\alpha\|_{\T}^2-\frac{|\xi|_2^2}{2}\right)d\xi = \sum_{s=-\infty}^{\infty}\int_{\pi s}^{\pi(s+1)}\exp\left(-\lambda \|\xi w+\alpha\|_{\T}^2-\frac{|\xi|_2^2}{2}\right)d\xi
\\ & \le \sum_{s=-\infty}^{\infty}e^{-\frac{\pi^2 s^2}2}\int_{\pi s}^{\pi(s+1)}\exp\left(-\lambda \|\xi w+\alpha\|_{\T}^2\right)d\xi \le \left(\sum_{s=-\infty}^{\infty}e^{-\frac{\pi^2s^2}{2}}\right)\frac{20( w +1\big)}{w\,\sqrt {1+\lambda}}
\\ & \le \frac{40(w +1)}{w\,\sqrt {1+\lambda}},
\end{align*}
which completes the proof.
\end{proof}

We are now in a position to prove Proposition~\ref{prop hyper}.

\begin{proof}[Proof of Proposition~\ref{prop hyper}]
By~\eqref{int v} and~\eqref{growth eta}, we have
\begin{eqnarray}\label{from esseen}
\I\left(X_V\right) \le \int_{\R^d}\exp\Bigg(-c_{\eta}\sum_{j=1}^n \|\langle v_j, \xi\rangle\|_{\T}^2-\frac{|\xi|_2^2}{2}\Bigg)d\xi.
\end{eqnarray}
Assume first that $k = d\ell$, $\ell\in \mathbb N$. The general case will be considered at the end of the proof. Let $v_{0, 1},\dots, v_{0,\ell}$ be $\ell$ elements of $V$, and set $V^{(1)} = V\setminus\{v_{0, 1},\dots, v_{0,\ell}\}$. Then, we can write
\begin{align}\label{want holder}
\nonumber & \int_{\R^d}\exp\Bigg(-c_{\eta}\sum_{j=1}^n \|\langle v_j, \xi\rangle\|_{\T}^2-\frac{|\xi|_2^2}{2}\Bigg)d\xi
\\ \nonumber & = \int_{\R^d}\exp\Bigg(-c_{\eta}\sum_{u\in V^{(1)}}\|\langle v,\xi\rangle\|_{\T}^2\Bigg)\prod_{j=1}^\ell\exp\Bigg(-c_{\eta}\|\langle v_{0,j}, \xi\rangle\|_{\T}^2\Bigg)\exp\Bigg(-\frac{|\xi|_2^2}{2}\Bigg)d\xi
\\ & = \int_{\R^d}\prod_{j=1}^\ell\exp\left(-c_{\eta}\|\langle v_{0,j}, \xi\rangle\|_{\T}^2-\frac 1 {\ell}\Bigg(c_{\eta}\sum_{u\in V^{(1)}}\|\langle v,\xi\rangle\|_{\T}^2+\frac{|\xi|_2^2}{2}\Bigg)\right)d\xi.
\end{align}
It follows from H\"older's inequality that if $f_1,\dots,f_\ell$ are positive functions, then 
\[\int_{\R^d}\prod_{j=1}^\ell f_j \le \prod_{j=1}^d\left(\int_{\R^d}f_j^\ell\right)^{\frac 1 \ell}.\]
In particular, there exists $j_0 \in \{1,\dots, \ell\}$ such that
\[\int_{\R^d}\prod_{j=1}^\ell f_j \le \int_{\R^d}f_{j_0}^\ell.\]
Therefore, applying H\"older's inequality to the right side of~\eqref{want holder}, we conclude that there exists an index $j_0\in \{1,\dots,\ell\}$ such that
\begin{multline}\label{holder gives}
\int_{\R^d}\prod_{j=1}^\ell\exp\left(-c_{\eta}\|\langle v_{0,j}, \xi\rangle\|_{\T}^2-\frac 1 {\ell}\Bigg(c_{\eta}\sum_{u\in V^{(1)}}\|\langle v,\xi\rangle\|_{\T}^2+\frac{|\xi|_2^2}{2}\Bigg)\right)d\xi
\\
\le \int_{\R^d}\exp\Bigg(-c_{\eta}\sum_{v\in V^{(1)}}\|\langle u, \xi\rangle\|_{\T}^2\Bigg)\exp\Bigg(-c_{\eta}\ell~\|\langle v_{0,j_0}, \xi \rangle\|_{\T}^2\Bigg)\exp\Bigg(-\frac{|\xi|_2^2}{2}\Bigg)d\xi.
\end{multline}
Plugging~\eqref{holder gives} into~\eqref{want holder} gives
\begin{multline*}
\int_{\R^d}\exp\Bigg(-c_{\eta}\sum_{j=1}^n\|\langle  v_j, \xi \rangle\|_{\T}^2-\frac{|\xi|_2^2}{2}\Bigg)d\xi
\\
\le \int_{\R^d}\exp\Bigg(-c_{\eta}\sum_{v\in V^{(1)}}\|\langle u, \xi\rangle\|_{\T}^2\Bigg)\exp\Bigg(-c_{\eta}\ell~\|\langle v_{0,j_0}, \xi \rangle\|_{\T}^2\Bigg)\exp\Bigg(-\frac{|\xi|_2^2}{2}\Bigg)d\xi.
\end{multline*}
Set $w_1 = v_{0, j_0}$. If $d = 1$, then we stop at this point. Otherwise, by the assumption on the vectors in $V$, we can choose $\ell$ elements $v_{1,1},\dots, v_{1,\ell}'$ of $V^{(1)}$ which lie at a distance at least $R$ from the $\spn\{w_1\}$. Set $V^{(2)} = V^{(1)}\setminus\{v_{1, 1},\dots, v_{1,\ell}\}$, and as before, find $j_1$ such that
\begin{align*}
&\int_{\R^d}\exp\Bigg(-c_{\eta}\sum_{j=1}^n\|\langle v_j, \xi \rangle\|_{\T}^2-\frac{|\xi|_2^2}{2}\Bigg)d\xi
\\ &\le\int_{\R^d} \exp\Bigg(-c_{\eta}\sum_{v\in V_2}\|\langle u,\xi \rangle\|_{\T}^2\Bigg) \exp\Bigg(-c_{\eta}\ell\|\langle w_1,\xi\rangle\|_{\T}^2\Bigg)
\\ & \qquad\quad \cdot \exp\Bigg(-c_{\eta}\ell\|\langle v_{1,j_1},\xi\rangle\|_{\T}^2\Bigg)\exp\Bigg(-\frac{|\xi|_2^2}{2}\Bigg)d\xi.
\end{align*}
Now, set $w_2 = v_{1, j_1}$, and repeat this procedure $d-1$ times, eventually obtaining
\begin{multline*}
\int_{\R^d}\exp\Bigg(-c_{\eta}\sum_{j=1}^n\|\langle v_j,\xi \rangle\|_{\T}^2\Bigg)d\xi
\\ \le\int_{\R^d}  \exp\Bigg(-c_{\eta}\sum_{v\in V^{(d)}}\|\langle u,\xi\rangle\|_{\T}^2\Bigg) \exp\Bigg(-c_{\eta}\ell\sum_{j=1}^d\|\langle w_j,\xi \rangle\|_{\T}^2\Bigg)\exp\Bigg(-\frac{|\xi|_2^2}{2}\Bigg)d\xi,
\end{multline*}
for some $w_1,\dots, w_d$, now with the property that 
\begin{align}\label{basis large dist}
\inf_{1 \le j \le n}\Big[\dist_2\big(w_j,\spn\{w_1,\dots, w_{j-1}\}\big)\Big]\ge R,
\end{align} 
for all $1\le j\le d$, and $V^{(d)}$ is a subset of $V$ with at least $n-k$ vectors. Note also that we can choose $w_1$ such that $|w_1|_2 \ge R$. In addition, the following trivial bound holds,
\begin{align*}
\exp\Bigg(-c_{\eta}\sum_{v\in V^{(d)}}\|\langle u,\xi\rangle\|_{\T}^2\Bigg)\le 1.
\end{align*}
Thus, we have
\begin{align}\label{ineq on wj's}
\int_{\R^d}\exp\Bigg(-c_{\eta}\sum_{j=1}^n\|\langle v_j, \xi \rangle\|_{\T}^2-\frac{|\xi|_2^2}{2}\Bigg)d\xi \le \int_{\R^d}  \exp\Bigg(-c_{\eta}\ell \sum_{j=1}^d\|\langle w_j, \xi\rangle\|_{\T}^2-\frac{|\xi|_2^2}{2}\Bigg) d\xi.
\end{align}
To bound the right side of \eqref{ineq on wj's}, use induction on $d$. Consider first the case $d=1$. By Lemma \ref{lem TV},
\begin{align}\label{case 1 dim}
\int_{\R}\exp\left(-c_{\eta}\ell\|\xi w+\alpha\|_{\T}^2-\frac{|\xi|_2^2}{2}\right)d\xi \le \frac{40(|w|+1)}{|w|\, \sqrt{1+c_{\eta}\ell}}.
\end{align}
In the case $d=1$, we the assumption on the vectors $\{v_1,\dots,v_n\}$ implies that $|w|\ge R$, and so~\eqref{case 1 dim} gives
\begin{align}\label{case 1 dim with r}
\int_{\R}\exp\left(-c_{\eta}\ell\|\xi w+\alpha\|_{\T}^2-\frac{|\xi|_2^2}{2}\right)d\xi \le \frac{40(R+1\big)}{R\, \sqrt{1+c_{\eta}\ell}}.
\end{align}
To handle the general case, use Fubini's Theorem and induction on $d$. By following a Gram-Schmidt process, find an orthonormal basis $\{e_1,\dots, e_d\}$ of $\R^d$, such that $\spn\{w_1,\dots, w_j\} =\spn\{e_1,\dots, e_j\}$, for all $1\le j\le d$. Suppose that the desired claim holds for dimension $d-1$, and for a vector $\xi\in\R^d$, write $\xi =\xi'+\xi_d e_d$, where $\xi'\in\spn\{e_1,\dots, e_{d-1}\}$ and $\xi_d\in\R$. This gives
\begin{align*}
\sum_{j=1}^d\|\langle w_j, \xi\rangle\|_{\T}^2 &  = \sum_{j=1}^{d-1}\|\langle w_j, \xi' \rangle\|_{\T}^2 +\|\xi_d\langle w_d, e_d\rangle\|_{\T}^2.
\end{align*}
Note that~\eqref{basis large dist} implies that 
\begin{align}\label{large dist d dim}
|\langle w_d, e_d\rangle| \ge R.
\end{align}
Thus, we have
\begin{eqnarray}\label{bound case k integer}
\nonumber &&\int_{\R^{d}} \exp\Bigg(-c_{\eta}\ell\sum_{j=1}^d\|\langle w_j, \xi \rangle\|_{\T}^2 - \frac{|\xi|_2^2}{2}\Bigg)d\xi 
\\ \nonumber && =\int_{\R^{d-1}}
\Bigg[\exp\Bigg(-c_{\eta}\ell\sum_{j=1}^{d-1}\|\langle w_j, \xi' \rangle\|_{\T}^2-\frac{|\xi'|_2^2}{2}\Bigg)
\int_{\R}\exp\Bigg(-c_{\eta}\ell\|\xi_d\langle w_d, e_d \rangle\|_{\T}^2-\frac{\xi_d^2}{2}\Bigg)
d\xi_d\Bigg]d\xi'
\\  \nonumber && \stackrel{\eqref{case 1 dim with r}\wedge\eqref{large dist d dim}}{\le} \frac{40(R+1)}{R\, \sqrt{1+c_{\eta}\ell}}\int_{\R^{d-1}}
\exp\Bigg(-c_{\eta}\ell\sum_{j=1}^{d-1}\|\langle w_j, \xi' \rangle\|_{\T}^2-\frac{|\xi'|_2^2}{2}\Bigg)d\xi' 
\\ && \qquad\stackrel{(*)}{\le}  \left(\frac{40(R+1)}{R\, \sqrt{1+c_{\eta}\ell}}\right)^d,
\end{eqnarray}
where in ($*$) we used the induction hypothesis. Combining~\eqref{from esseen},~\eqref{ineq on wj's} and~\eqref{bound case k integer} gives that when $k=d\ell$,
\begin{align*}
\mathcal I(X_V) \le \left(\frac{40(R+1)}{R\, \sqrt{1+c_{\eta}\ell}}\right)^d.
\end{align*} 
In general, assume that $d(\ell-1) \le k \le d\ell$. Then we know that there are at least $n-d\ell$ vectors which satisfy $\dist_2(v_j,H)\ge R$ for every hyperplane $H\subseteq \R^d$. In such case, since $\ell \ge k/d$ we have
\begin{align*}
\mathcal I(X_V) \le \left(\frac{40(R+1)}{R\, \sqrt{1+c_{\eta}\ell}}\right)^d \le \left(\frac{40(R+1)}{R\, \sqrt{1+c_{\eta}(k/d)}}\right)^d = \left(40\frac{R+1}{R}\sqrt{\frac d {d+c_{\eta}k}}\,\right)^d,
\end{align*}
which completes the proof.
\end{proof}

Now with Proposition~\ref{prop hyper} in hand, we can prove Theorem~\ref{thm hyper}.

\begin{proof}[Proof of Theorem~\ref{thm hyper}]
By~\eqref{esseen k}, we have
\begin{align*}
\rho_R^K(X_V) \le \kappa(K)^d\I\left(\frac1 R\, X_V\right) = \kappa(K)^d\I\left( X_{V_R}\right),
\end{align*}
where $V_R = \{R^{-1}v_1,\dots,R^{-1}v_n\}$. In particular, by the assumptions on the set $V$, we know that for at least $n-k$ vectors $v'_j\in V_R$, we have $\dist_2(v'_j, H) \ge 1$. Hence, using Proposition~\ref{prop hyper} with $R=1$, we have
\begin{align*}
\rho_R^K(X_V) \le \kappa(K)^d\left(80\sqrt{\frac{d}{d+c_{\eta}k}}\,\right)^d = \big(80\kappa(K)\big)^d\left(\frac d {d+c_{\eta}k}\right)^{d/2},
\end{align*}
which completes the proof.
\end{proof}

\section{Proof of Theorem \ref{thm gap}}\label{sec proof gap}

Recall again that by \eqref{esseen k}, we have
\begin{align*}
\rho_R^K(X_V) \le \big(\kappa(K)R\big)^d\int_{\R^d}\prod_{j=1}^n \big| \E \exp\left(i\left\langle \eta_j v_j, \xi\right\rangle\right)\big|\, e^{-\frac{R^2|\xi|_2^2}{2}} d\xi.
\end{align*}
For a random variable $\eta$ and a real number $a$, define
\begin{align}\label{eta norm}
\|a\|_{\eta} = \frac 2 {\pi}\left(\E\left\|a(\eta_1-\eta_2)\right\|_{\T}^2\right)^{1/2},
\end{align}
where $\eta_1$ and $\eta_2$ are independent copies of $\eta$, and $\|\cdot\|_{\T}$ is as defined in~\eqref{def norm}. The constant $\frac 2 {\pi}$ is simply a normalisation constant which makes some of the computations simpler.
Now, it is easy to show that we have
\begin{align*}
\big|\E\exp\big(ia\eta\big)\big|^2 = \big|\E\cos\big(a(\eta_1-\eta_2)\big)\big| \le 1- \frac 2{\pi^2}\E\left\|a(\eta_1-\eta_2)\right\|_{\T}^2
\le \exp\Big(-\frac 1 2 \|a\|_{\eta}^2\Big).
\end{align*}
Hence,~\eqref{esseen k} implies in fact that we have
\begin{align}\label{esseen eta norm}
\rho_R^K(X_V) =\rho_1^K(X_{V_R}) \le \kappa(K)^d\int_{\R^d}\exp\Bigg(-\frac 1 2 \sum_{v\in V_R}\|\langle v,\xi\rangle\|_{\eta}^2 -\frac 1 2 |\xi|_2^2\Bigg)d\xi.
\end{align}
The first step in the proof of Theorem \ref{thm gap} is to find a large subset of $\R^d$ on which the sum $\sum_{v\in V_R}\|\langle v,s\rangle\|_{\eta}^2$ is relatively small. Such a set should then have some arithmetic structure. The proof is similar to \cite{NV11} and we present it for the sake of completeness, while making the required modifications.

\begin{prop}\label{prop subset}
Assume that $A>0$ is an absolute constant, and assume that for every $n$ sufficiently large, we have
\[\rho_R^K(X_V) \ge n^{-A}.\] 
Then there exists a positive integer $m$ satisfying
\begin{align*}
m \le 10\sqrt{d\log\left(\kappa(K)n^A\right)},
\end{align*}
such that for any sufficiently large integer $N$, there exists a finite set $S\subseteq \R^d$ 
\begin{align*}
|S| \ge \left(\frac{N}{2\sqrt d\kappa(K)}\right)^d\rho_R^K(X_V),
\end{align*}
and if we let $V_R = \{R^{-1}v_1,\dots ,R^{-1}v_n\}$, then there exists a real number $\alpha$ which depends only on $C_{\eta}$ such that the set $S$ satisfies
\begin{align*}
\frac 1 {|S|}\sum_{s\in S}\Bigg[\sum_{v\in V_R}\|\alpha\langle v,s\rangle\|_{\eta}^2\Bigg] \le 8\pi^2 m.
\end{align*}
\end{prop}

\begin{proof}
First, notice that we have $\rho^K_R(X_V) = \rho^K_1(X_{V_R})$. Let $\rho = \rho_1^K(X_{V_R})$. We have
\begin{eqnarray*}
\int_{|\xi|_2> M}\exp\left(-\frac 1 2\sum_{v\in V_R}\|\langle v,\xi\rangle\|_{\eta}^2-\frac 1 2|\xi|_2^2\right)d\xi \le \int_{|\xi|_2> M}\exp\left(-\frac 1 2|\xi|_2^2\right)d\xi \le (2\pi)^{d/2}e^{-M^2/4},
\end{eqnarray*}
where we used the fact the a normal distribution has a subgaussian tail. In particular, assuming that $n\ge \kappa(K)^{-1/A}$, and choosing $M$ to be
\[M = \sqrt{-4\log\left(\frac{n^{-A}(2\pi)^{d/2}}{2\kappa(K)^d}\right)} \le 10\sqrt{d\log \big(\kappa(K)n^A\big)},\]
it follows that
\begin{align}\label{bound tail}
\int_{|\xi|_2> M}\exp\left(-\frac 1 2\sum_{v\in V_R}\|\langle v,\xi\rangle\|_{\eta}^2-\frac 1 2|\xi|_2^2\right)d\xi  \le (2\pi)^{d/2}e^{-M^2/2} = \frac{n^{-A}}{2\kappa(K)^d} \le \frac{\rho}{2\kappa(K)^d} .
\end{align}
Using~\eqref{esseen eta norm}, it follows that
\begin{align}\label{bound int everything}
\frac{\rho}{\kappa(K)^d} \le \int_{\R^d}\exp\left(-\frac 1 2\sum_{v\in V_R}\|\langle v,\xi\rangle\|_{\eta}^2-\frac 1 2|\xi|_2^2\right)d\xi.
\end{align}
Thus, combining~\eqref{bound tail} and~\eqref{bound int everything}, we have
\begin{align}\label{lower bound with m}
\int_{|\xi|_2\le M}\exp\left(-\frac 1 2\sum_{v\in V_R}\|\langle v,\xi\rangle\|_{\eta}^2-\frac 1 2|\xi|_2^2\right)d\xi \ge \frac{\rho}{2\kappa(K)^d}.
\end{align}
For $m \in \{0,1,\dots, M\}$, define
\[T_m = \left\{\xi \in \R^d~\left|~\sum_{v\in V_R}\|\langle v,\xi\rangle\|_{\eta}^2+ |\xi|_2^2 \le m\right.\right\}.\]
Then~\eqref{lower bound with m} implies that
\begin{align*}
\sum_{m=0}^M\mu_d(T_m)e^{-m/2} \ge \frac{\rho}{2\kappa(K)^d},
\end{align*}
where $\mu_d(\cdot)$ denotes the Lebesgue measure on $\R^d$. It follows that there exists $m\le M$ such that $\mu_d(T_m) \ge \frac{\rho e^{m/4}}{2\kappa(K)^d}$. Now, since clearly we have $T_m \subseteq B_2^d(0,\sqrt m)$, by the pigeon-hole principle, there exists $x\in \R^d$ such that $B_2^d(x,1/2)\subseteq B_2^d(0,\sqrt m)$ and
\begin{align}\label{lower bound measure}
\mu_d\big(B_2(x,1/2)\cap T_m\big) \ge \mu_d(T_m)m^{-d/2} \ge \frac{\rho e^{m/4}m^{-d/2}}{2\kappa(K)^d},
\end{align}
where we recall that for $\alpha >0$, $B_2^d(x,\alpha) = x+\alpha\,B_2^d$. Next, let $\xi_1,\xi_2\in B(x,1/2)\cap T_m$. Since $\|\cdot\|_{\eta}$ satisfies the triangle inequality, we have
\[\sum_{v\in V_R}\|\langle v,\xi_1-\xi_2\rangle\|_{\eta}^2 \le 2\sum_{v\in V_R}\|\langle v,\xi_1\rangle\|_{\eta}^2+2\sum_{v\in V_R}\|\langle v,\xi_2\rangle\|_{\eta}^2 \stackrel{(*)}{\le} 2m+2m = 4m,\]
where in ($*$) we used the fact that $\xi_1,\xi_2\in T_m$. Now, we have $\xi_1-\xi_2 \in B_2^d$. Also, we have
\begin{align}\label{inclusion difference}
\mu_d\Big(B_2^d(x,1/2)\cap T_m - B_2^d(x,1/2)\cap T_m\Big) \ge \mu_d\Big(B_2^d(x,1/2)\cap T_m\Big).
\end{align}
Hence, defining 
\begin{align*}
T = \left\{\xi\in B_2^d ~\left|~ \sum_{v\in V_R}\|\langle v,\xi\rangle\|_{\eta}^2 \le 4m\right.\right\},
\end{align*}
it follows that we have
\begin{eqnarray*}
\mu_d(T) \stackrel{\eqref{lower bound measure}\wedge \eqref{inclusion difference}}{\ge} \frac{\rho e^{m/4}m^{-d/2}}{2\kappa(K)^d}.
\end{eqnarray*} 
Now, use the fact that for every positive integer $m$, $e^{m/4}m^{-d/2} \ge e^{d/2}(2d)^{-d/2}$ to conclude that
\begin{align}\label{size T}
\mu_d(T) \ge \frac 1 2\left(\frac{\sqrt e}{\sqrt{2d}\kappa(K)}\right)^d\rho \ge \frac{\rho}{\left(2\sqrt d \kappa(K)\right)^d}.
\end{align}
Next, for a given positive integer $N$, let $B_0$ be the discrete box 
\[B_0 = \big\{(k_1/N,\dots,k_d/N)\in \R^d ~\big| -N\le k_j \le N\big\}.\]
Consider all the boxes $x+B_0$ with $x\in [0,1/N]$. Since
\[\lim_{N\to \infty}\frac 1 {N^d}\E\big|(x+B_0)\cap T\big| = \mu_d(T),\]
it follows that there exists $x_0\in \R^d$ and a positive integer $N$, such that 
\[\big|(x_0+B_0)\cap T\big| \ge N^d\mu_d(T).\] 
Fix $\xi_0\in x_0+B_0$. For any $\xi\in (x_0+B_0)\cap T$, we have
\begin{align}\label{upper bound sum difference}
\sum_{v\in V_R}\|\langle v,\xi_0-\xi\rangle\|_{\eta}^2 \le 2\sum_{v\in V_R}\|\langle v,\xi_0\rangle\|_{\eta}^2+2\sum_{v\in V_R}\|\langle v,\xi\rangle\|_{\eta}^2 \le 16m.
\end{align}
We have, 
\[\xi_0-\xi \in B_0-B_0 = \big\{(k_1/N,\dots,k_d/N)\in \R^d~\big| -2N \le k_j \le 2N\big\}.\] 
This means that there exists a subset $S\subseteq (B_0-B_0)\cap T$ of size at least $N^d\mu_d(T)$ such that for any $s\in S$,
\begin{align*}
\sum_{v\in V_R}\|\langle v,s\rangle\|_{\eta}^2 \stackrel{\eqref{upper bound sum difference}}{\le} 16m.
\end{align*}
Hence, we have
\begin{align}\label{upper bound average}
\frac 1 {|S|}\sum_{s\in S}\Bigg[\sum_{v\in V_R}\|\langle v,s\rangle\|_{\eta}^2\Bigg] \le 16 m.
\end{align}
By combining~\eqref{upper bound average} with \eqref{eta norm}, it follows that
\begin{align}\label{new bound average}
\frac 1 {|S|} \sum_{s\in S}\Bigg[\sum_{v\in V_R}\E\|(\eta_1-\eta_2) \langle v,s\rangle \|_{\T}^2\Bigg] \le 4\pi^2m.
\end{align}
Now, there exists $\alpha \in [1,C_{\eta}]$, such that
\begin{align}\label{lower bound expect}
\E\|(\eta_1-\eta_2) \langle v,s\rangle \|_{\T}^2 \ge \p\big(1\le |\eta_1-\eta_2| \le C_{\eta}\big)\, \|\alpha \langle v,s\rangle \|_{\T}^2 \stackrel{\eqref{assume eta}}{\ge} \frac 1 2 \|\alpha \langle v,s\rangle \|_{\T}^2.
\end{align}
Therefore, combining~\eqref{new bound average} and~\eqref{lower bound expect}, we have
\begin{align*}
\frac 1 {|S|} \sum_{s\in S}\Bigg[\sum_{v\in V_R}\|\alpha \langle v,s\rangle \|_{\T}^2\Bigg] \le 8\pi^2m,
\end{align*}
which completes the proof.
\end{proof}

Fix $\eps \in (0,1)$. Let $n'$ be an integer between $n^{\eps}$ and $n$. Let $S\subseteq \R^d$ be the set from Proposition~\ref{prop subset}. Say that a vector $v\in V_R$ is said to be \emph{bad} if
\begin{align*}
\frac 1 {|S|}\sum_{s\in S}\|\alpha \langle v,s\rangle\|_{\T}^2 \ge \frac{8\pi^2m}{n'}.
\end{align*}
Let $V_R'$ denote all the vectors in $V_R$ which are not bad. It follows that $|V_R'| \ge n-n'$. Recall also that if $A\subseteq \R^d$ and $k$ is a positive integer, then we denote
\begin{align*}
kA = \Big\{\sum_{j=1}^ka_j~\Big|~a_j \in A\Big\}.
\end{align*}

\begin{lem}\label{lem dual}
Let $n'\in [n^{\eps},n]$ and choose 
\begin{align}\label{choice k}
k = \sqrt{\frac{n'}{64\pi^2 m}}\, .
\end{align}
If $N$ is sufficiently large, then
\begin{align*}
\mu_d\left(k(V_R'\cup\{0\})+B_{\infty}\left(0,\frac 1 {256d\alpha}\right)\right) \le 36\pi^2\left(\frac{2\sqrt d\kappa(K)}{\alpha}\right)^d\left(\frac 1 {\rho_R^K(X_V)}\right).
\end{align*}
\end{lem}

\begin{proof}
Let $v\in V_R'$. By definition of vectors which are not bad, we have
\begin{align*}
\sum_{s\in S}\|\alpha\langle v,s\rangle\|_{\T}^2 \le \frac{8\pi^2m|S|}{n'}.
\end{align*}
Since $\|\cdot\|_{\T}$ satisfies the triangle inequality, for any $w\in k(V_R'\cup\{0\})$ we have
\begin{align*}
\sum_{s\in S}\|\langle s,\alpha w\rangle\|_{\T}^2 \le k^2\sup_{v\in V_R'}\Bigg[\sum_{s\in S}\|\langle s,\alpha v\rangle\|_{\T}^2\Bigg] \le \frac{8\pi^2m|S|k^2}{n'} \stackrel{\eqref{choice k}}{\le} \frac{|S|}{4}.
\end{align*}
Since for any real number $a$, $\cos(2 a) \ge 1-2\|a\|_{\T}^2$,
\begin{align*}
\sum_{s\in S}\cos\big(2\langle s,\alpha w\rangle\big) \ge |S| -\sum_{s\in S}2\|\langle s,\alpha v\rangle\|_{\T}^2 \ge  \frac{|S|}{2}.
\end{align*}
Note that if $\|x\|_{\infty}\le \pi/256d$ and $\|s\|_{\infty} \le 2$, then $\cos\big(2\langle x,s\rangle\big) \ge 1/2$ and $\sin\big(2\langle x,s\rangle\big) \le 1/12$. Thus, for any $x$ with $\|x\|_{\infty} \le \pi/256d$,
\begin{align*}
\sum_{s\in S}\cos\big(2\langle s,\alpha w+x\rangle\big)  & = \sum_{s\in S}\cos\big(2\langle s,\alpha w\rangle\big)\cos\big(2\langle s,x\rangle\big) -\sum_{s\in S}\sin\big(2\langle s,\alpha w\rangle\big)\sin\big(2\langle s,x\rangle\big)
\\ &\ge \frac{|S|}4-\frac{|S|}{12} = \frac{|S|}{6}.
\end{align*}
On the other hand, we have
\begin{align*}
\int_{[0,\pi N]^d}\left(\sum_{s\in S}\cos\big(2\langle s,x\rangle\big)\right)^2dx \le \sum_{s_1,s_2\in S}\int_{x\in [0,\pi N]^d}\exp\big(2 i \langle s_1-s_2,x\rangle \big)dx = \pi^2N^2|S|,
\end{align*}
and so
\begin{align*}
\mu_d\left(\left\{x\in[0,\pi N]^d ~ \Bigg| ~\left(\sum_{s\in S}\cos\big(2\langle s,x\rangle\big)\right)^2\ge \left(\frac{|S|}{6}\right)^2\right\}\right) \le \frac{\pi^2N^2|S|}{(|S|/6)^2} = 36\pi^2\frac{N^d}{|S|}.
\end{align*}
Let $V''_R= k\big(V_R'\cup \{0\}\big)$. If $N$ is chosen such that $\alpha V_R''+B_{\infty}(0,\pi/256d) \subseteq [0,\pi N]^d$, then it follows that
\begin{align*}
\mu_d\left(\alpha V_R''+B_{\infty}\left(0,\frac {\pi} {256d}\right)\right) & \le \mu_{d}\left(\left\{x\in[0,\pi N]^d~\Bigg|~\left(\sum_{s\in S}\cos\big(2\pi\langle s,x\rangle\big)\right)^2 \ge \left(\frac{|S|}{6}\right)^2\right\}\right) 
\\ & \le 36\pi^2\frac{N^d}{|S|}.
\end{align*}
Now, let $N$ be large enough so that Proposition~\ref{prop subset} holds. The result now follows from homogeneity of $\mu_d$ and Proposition~\ref{prop subset}.
\end{proof}

\begin{remark}
Since $k = \sqrt{\frac{n'}{64\pi^2m}}$, using the estimate on $m$ from Proposition~\ref{prop subset} it follows that
\begin{align*}
\sqrt{\frac{n'}{640\pi^2\sqrt{d\log\left(\kappa(K)n^A\right)}}} \le k \le \sqrt{n'}.
\end{align*}
\end{remark}

The remaining main tools in the proof of Theorem \ref{thm gap} the are the following.

\begin{thm}[\cite{NV11}]\label{thm inverse}
Assume that $X$ is a discrete subset of a torsion free group. Assume that there exists an integer $k$ such that $|kX| \le k^{\gamma}|X|$ for some positive number $\gamma$. Then there exists a proper GAP $Q$ with rank $\rank(Q)\le C\gamma$ and cardinality $|Q| \le C(\gamma)\, k^{-\rank(Q)}|kX|$, such that $X\subseteq Q$.
\end{thm}

We are now in a position to prove Theorem \ref{thm gap}.

\begin{proof}[Proof of Theorem \ref{thm gap}]
\noindent \emph {Proof of Part~\ref{part 1}.}
Let $D= 512d\alpha $. Also, as before, let $\rho = \rho_R^K(X_V) = \rho_1^K(X_{V_R})$, where $V_R = \{R^{-1}v_1,\dots,R^{-1}v_n\}$. For $v'\in V_R'$, find $z\in \Z^d$ such that $z/Dk$ is the closest vector to $v'$ in the $\ell_{\infty}$ norm. That is, we can choose $z\in \Z^d$ such that
\begin{align}\label{close vec infty}
\left\|v'-\frac z {Dk}\right\|_{\infty} \le \frac 1 {Dk}.
\end{align}
Denote the set of all $z$'s satisfying~\eqref{close vec infty} by $F $. Let $\sum_{j=1}^kv_j'\in k(V_R'\cup \{0\})$. For each $j\le k$, let $z_j$ be the approximation to $v_j'$ as in~\eqref{close vec infty}. Then,
\begin{align*}
\left\|x-\sum_{j=1}^kv_j'\right\|_{\infty} \le \left\|x-\sum_{j=1}^k\frac {z_j} {Dk}\right\|_{\infty}+\sum_{j=1}^k\left\|x-\frac {z_j} {Dk}\right\|_{\infty} \le \left\|x-\sum_{j=1}^k\frac {z_j} {Dk}\right\|_{\infty}+\frac 1 D,
\end{align*}
which implies that
\begin{align*}
k\left(\frac 1 {Dk}F \right)+B_{\infty}\left(0,\frac 1 D\right) \subseteq k\big(V_R'\cup \{0\}\big)+B_{\infty}\left(0,\frac 2 D\right).
\end{align*}
Recall that here $\frac 1 {Dk}F  = \{x/(Dk)~|~x\in F \}$. Thus, by the choice of $D$ and Lemma \ref{lem dual}, we have
\begin{align*}
\mu_d\left(k\left(\frac 1 {Dk}F \right)+B_{\infty}\left(0,\frac 1 {D}\right)\right)  \le 36\pi^2\left(\frac{2\sqrt d\kappa(K)}{\alpha }\right)^d\rho^{-1}.
\end{align*}
Therefore, by homogeneity,
\begin{align*}
\mu_d\big(kF +B_{\infty}(0,k)\big) \le 36\pi^2\left(\frac{2\sqrt d Dk~\kappa(K)}{\alpha }\right)^d\rho^{-1} \le \big(C \kappa(K)d^{3/2}\big)^d k^d\rho^{-1},
\end{align*}
which implies that there exists an absolute constant $C$ such that
\begin{align}\label{size k sum}
\left|k\left(F +\{-1,1\}^d\right)\right| \le \big(C\kappa(K)d^{3/2}\big)^d k^d\rho^{-1}.
\end{align}
Using the choice $k = \sqrt{\frac{n'}{64\pi^2 m}}$, the fact that $n'\in [n^{\eps},n]$ and the fact that $m \le 10\sqrt{d\log \left(n^A\kappa(K)\right)}$, it follows that for sufficiently large $n$,
\begin{align*}
\frac 1 {8\pi}\left(\frac{n^{\eps}}{Ad\log(\kappa(K))}\right)^{1/4} \le \sqrt{\frac{n^{\eps}}{64\pi^2 m}} \le k \le \sqrt{n'} \le \sqrt n.
\end{align*}
Hence, we have
\begin{align*}
\rho^{-1} \le n^A \le \big(Ad\log(\kappa(K))\big)^{\frac{A}{\eps}}\big(8\pi k\big)^{\frac{4A}{\eps}}.
\end{align*}
Plugging this into \eqref{size k sum},
\begin{align*}
\left|k\left(F +\{-1,1\}^d\right)\right| & \le \big(C \kappa(K)d^{3/2}\big)^{d}(8\pi)^{\frac{4A}{\eps}}\big(Ad\log(\kappa(K))\big)^{\frac{A}{\eps}}k^{d+\frac{4A}{\eps}}
\\ & = k^{d+\frac{4A}{\eps}+\frac{\log (C \kappa(K)) + \frac{4A}{\eps}\log(8\pi) + \frac{A}{\eps}\log(Ad\log(\kappa(K))))}{\log k}}.
\end{align*}
Assuming that $d$, $A$, $\eps$ and $\kappa(K)$ are given constant and $k\to \infty$, it follows that
\begin{align*}
d+\frac{4A}{\eps}+\frac{\log(C \kappa(K)) + \frac{4A}{\eps}\log(8\pi) + \frac{A}{\eps}\log(Ad\log(\kappa(K))))}{\log k} \le d+\frac{8A}{\eps}.
\end{align*}
Using the trivial bound $\big|F +\{-1,1\}^d\big| \ge 1$, and choosing $\gamma = d+\frac{8A}{\eps}$, gives
\begin{align*}
\left|k\left(F +\{-1,1\}^d\right)\right| \le k^{\gamma} \left|F +\{-1,1\}^d\right|.
\end{align*}
Now use Theorem \ref{thm inverse} with the set $F +\{-1,1\}^d$ to deduce that there exists a proper GAP 
\[Q' = \Big\{\sum_{j=1}^{r} x_jg_j ~\Big|~x_j\in \Z,~ |x_j| \le L_j\Big\},\]
which contains $F +\{-1,1\}^d$, and has rank
\begin{align*}
\rank(Q') \le  C\left(d+\frac{A}{\eps}\right),
\end{align*}
and cardinality
\begin{eqnarray*}
|Q'|  &\le  & C(\gamma)k^{-\rank(Q')}\left|k\left(F +\{-1,1\}^d\right)\right| 
\\& \stackrel{\eqref{size k sum}}{\le} & C(A, d, \eps)\big(C \kappa(K)d^{3/2}\big)^d k^{d-\rank(Q')}\rho^{-1}
\\&  \le & C\left(A, d, \eps\right)\big(C \kappa(K)d^{3/2}\big)^d (n')^{\frac{d-\rank(Q')}{2}}\rho^{-1}.
\end{eqnarray*}
Define
\begin{align*}
Q = \frac R{Dk}Q' = \left\{\frac {R}{Dk}\, q'~\Bigg|~ q'\in Q'\right\}.
\end{align*}
Then $Q$ has the same rank and cardinality of $Q'$. This completes the proof of Part~\ref{part 1}.

\noindent \emph{Proof of Part~\ref{part 2}.}
By \eqref{close vec infty} and the fact that $\frac 1 {\omega_{\infty}(K)}B_{\infty}^d\subseteq K$, we have for every $v'\in V_R'$,
\begin{align*}
\left\|v'-\frac z{Dk}\right\|_K \le \frac{\omega_{\infty}(K)}{Dk}.
\end{align*}
Since $V_R' \subseteq V_R = \{R^{-1}v_1,\dots,R^{-1}v_n\}$ and since $|V_R'| \ge n-n'$ it follows that for at least $n-n'$ elements of $V$, there exists $q\in Q$ with 
\[\|v-q\|_K \le \frac{\omega_{\infty}(K)R}{Dk},\]
which proves Part~\ref{part 2}.

\noindent \emph{Proof of Part~\ref{part 3}.}
Follows from the fact that $Q = \frac R{Dk}Q'$.

\noindent \emph{Proof of Part~\ref{part 4}.}
By \eqref{close vec infty} it follows that
\begin{align}\label{close vec}
\left\|v'-\frac z {Dk}\right\|_{K} \le \frac{\omega_{\infty}(K)}{Dk}.
\end{align}
Also, if $\|v'\|_{\infty} \le \frac 1 {Dk}$, can choose $z=0$ as the approximation in $\Z^d$. This means that whenever $\|v'\|_K \le \frac{\omega_{\infty}(K)}{Dk}$, we can choose $z=0$ as the approximation. Otherwise, if $\|v'\|_K \ge \frac{\omega_{\infty}(K)}{Dk}$, we have
\begin{align*}
\|z\|_{K} \stackrel{\eqref{close vec}}{\le} C_K\left(\omega_{\infty}(K)+Dk\|v'\|_K\right) \le 2C_KDk\|v'\|_K.
\end{align*}
Since $v'\in V_R' \subseteq V_R$, it follows that $\max_{z\in F }\|z\|_K \le \frac{2Dk}{R}\max_{v\in V}\|v\|_K$. Now by the results of~\cite{NV11, SV06, Tao10}, it follows that there exist $C=C(A,d,\eps)$ such that
\[kQ \subseteq C(A,d,\eps)\, \Big[k\big(F +\{-1,1\}^d\big)\Big].\] 
In particular, for every $1\le j \le r$, we have
\begin{align}\label{bound on generator}
\|g_j\|_K \le C(A,d,\eps)\, C_K^{k+1}\left(\max_{z\in F }\|z\|_K +\max_{z\in \{-1,1\}^d}\|z\|_K\right),
\end{align}
where we use the fact that in a quasi-normed space, we have
\[\left\|\sum_{j=1}^kv_j\right\|_K \le C_K^k\sum_{j=1}^k\|v_j\|_K.\]
Now,~\eqref{bound on generator} implies,
\begin{align*}
\|g_j\|_K \le C(A,d,\eps)\,C_K^{k+1}\left(\frac{Dk}{R}\max_{v\in V}\|v\|_K+\omega_{\infty}(K)\right),
\end{align*}
which completes the proof of Part~\ref{part 4} and of Theorem~\ref{thm gap}.
\end{proof}

\end{document}